\documentclass[10pt]{amsart}
\usepackage{amsmath}
\usepackage{amssymb}
\usepackage{amsfonts}
\usepackage{amsthm}
\usepackage{enumerate}
\usepackage[all]{xy}
\usepackage[latin1]{inputenc}
\usepackage{mathdots}
\usepackage{graphicx}
\usepackage{latexsym}
\usepackage{mathrsfs}
\usepackage{pgf}
\usepackage{braids}
\usepackage{floatrow}
\usepackage{subfig}

\usepackage{tikz}

\PassOptionsToPackage{usenames,dvipsnames,svgnames}{xcolor}

\input xy

\newtheorem{Theo}{Theorem}[section]

\newtheorem{Prop}[Theo]{Proposition}
\newtheorem{Cor}[Theo]{Corollary}
\newtheorem{Lemma}[Theo]{Lemma}

\theoremstyle{definition}

\newtheorem{Exam}[Theo]{Example}
\newtheorem{Remark}[Theo]{Remark}
\newtheorem{Remarks}[Theo]{Remarks}

\def\mystrut(#1,#2){\vrule height #1pt depth #2pt width 0pt}
\newcommand{\rep}{{\rm rep}}

\newcommand{\Hom}{{\rm Hom}}
\newcommand{\Ext}{{\rm Ext}}

\newcommand{\Z}{\mathbb{Z}}
\newcommand{\C}{\mathcal{C}}
\newcommand{\A}{\mathcal{A}}
\newcommand{\B}{\mathcal{B}}

\newcommand{\T}{\mathcal{T}}

\begin{document}

\author{Charles Paquette}
\address{Charles Paquette, Department of Mathematics and Computer Science, Royal Military College of Canada, Kingston, ON, K7K 7B4, Canada.}
\email{charles.paquette.math@gmail.com}
\thanks{The author is thankful to Henning Krause for pointing out Remark \ref{RemarkKrause}(1).}

\title[Generators in abelian categories]{Generators versus projective generators in abelian categories}

\begin{abstract}
Let $\mathcal{A}$ be an essentially small abelian category. We prove that if $\mathcal{A}$ admits a generator $M$ with ${\rm End}_{\mathcal{A}}(M)$ right artinian, then $\mathcal{A}$ admits a projective generator. If $\mathcal{A}$ is further assumed to be Grothendieck, then this implies that $\mathcal{A}$ is equivalent to a module category. When $\mathcal{A}$ is Hom-finite over a field $k$, the existence of a generator is the same as the existence of a projective generator, and in case there is such a generator, $\mathcal{A}$ has to be equivalent to the category of finite dimensional right modules over a finite dimensional $k$-algebra. We also show that when $\mathcal{A}$ is a length category, then there is a one-to-one correspondence between exact abelian extension closed subcategories of $\mathcal{A}$ and collections of Hom-orthogonal Schur objects in $\mathcal{A}$. \end{abstract}

\maketitle

\section{Introduction}

Let $\A$ be an abelian category. A natural and fundamental problem is to determine whether $\A$ is equivalent to a category of modules over a ring. It is well known that this is true if and only if $\A$ is co-complete and admits a compact projective generator, that is, an object $P \in \A$ which generates $\A$ (see below for the definition) and such that $\Hom_\A(P,-)$ is exact and commutes with arbitrary direct sums.

\medskip

In this paper, we consider the notion of a generator of $\A$. An object $M$ of $\A$ is a \emph{generator} of $\A$ if for any object $X$ of $\A$, we have an epimorphism $\oplus_{i \in I}M \to X$ where $I$ is some index set.
A (minimal) generator needs not, {\sl a priori}, be projective, since $\A$ does not necessarily have enough projective objects. In the first section of this paper, we will see that when $\A$ has a generator $M$ with ${\rm End}_\A(M)$ right artinian, then $\A$ also has a projective generator. In case $\A$ is further assumed to be Grothendieck, then it has to be equivalent to a module category over a right artinian ring. In case $\A$ is a length category or is Hom-finite over a field, then $\A$ is equivalent to the module category of finitely generated modules over a right artinian ring.

\medskip

When $\A$ is a length category, it need not have a generator. However, $\A$ has simple objects and these objects can be used to build all objects of $\A$ by successive extensions. Moreover, this set of objects need not be finite. In the second section, we consider Hom-orthogonal sets of Schur objects (or bricks) in $\A$ and prove that these are in bijection with the exact abelian extension-closed subcategories of $\A$. Finally, in the third section, we apply our results in the hereditary case, where the exact abelian extension closed subcategories are the same as the thick subcategories.

\medskip

The paper is self-contained and all proofs are elementary. The author has been informed by Henning Krause that some results of Section 1 can be derived by the well known Gabriel-Popescu theorem. An outline of Krause's argument will be given in Remark \ref{RemarkKrause}(1).

\section{Generators, projective generators and length categories}

Throughout, the symbol $\A$ always stands for an abelian category which is essentially small. We start by recalling some finiteness conditions on the objects of $\A$ and the notion of generator.

\medskip

An object $M \in \A$ is \emph{artinian} if any descending chain of subobjects of $M$ becomes stationary. The category $\A$ is \emph{artinian} if all objects of $\A$ are artinian. Similarly, an object $M \in \A$ is \emph{noetherian} if any ascending chain of subobjects of $M$ becomes stationary. The category $\A$ is \emph{noetherian} if all objects of $\A$ are noetherian. A non-zero object $X$ in $\A$ is called \emph{simple} or \emph{minimal} if it has no proper non-zero subobject. If $\A$ is both artinian and noetherian, then it is called a \emph{length category}. An object $X \in \A$ is of \emph{finite length} if it is both artinian and noetherian. Thus, for a finite length object $X$, there is a finite chain
$$0=X_n \subset X_{n-1} \subset \cdots \subset X_1 \subset X_0 = X$$
of subobjects of $X$ such that the quotients $X_{i-1}/X_{i}$ are simple for all $1 \le i \le n$. Such a chain is called a \emph{composition series} of $X$ and the length $n$ of this series is uniquely determined by $X$ and called the \emph{length} of $X$. This is known as (the categorical version of) the Jordan-Hölder theorem.

\medskip

A \emph{generator} of $\A$ is an object $M$ of $\A$ such that for any $X \in \A$, there is an epimorphism $\textstyle{\bigoplus}_{i \in I} \to X$ for some index set $I$. We will see that $\A$ having a generator $M$ with ${\rm End}_A(M)$ right artinian imposes many restrictions on $\A$. We start with the following lemma.

\begin{Lemma} \label{PropFinite}If $\A$ admits a generator and is artinian, then $\A$ has finitely many non-isomorphic simple objects.
\end{Lemma}

\begin{proof} Assume that $\A$ is artinian and admits a generator $M$. Assume to the contrary that $\A$ has infinitely many simple objects, up to isomorphism.  Let $M=M_1 \oplus \cdots \oplus M_n$ be a decomposition of $M$ into indecomposable direct summands (which is guaranteed by $\A$ being artinian). Let $S$ be a simple object in $\A$. Since $S$ is simple, there exists $1 \le i \le n$ such that there is an epimorphism $M_i \to S$. Therefore, we may assume that there is $1 \le i \le n$ such that $M_i$ has infinitely many non-isomorphic simple quotients. Let
$\{S_j\}_{j \ge 1}$ be such an infinite collection of non-isomorphic simple quotients of $M_i$. For each $j$, let $K_j$ denote the kernel of a projection $M_i \to S_j$.
Consider the diagram
$$\xymatrix{0 \ar[r] & K_{1} \ar[r]^{f_1} & M_i \ar@{=}[d] \ar[r]^{p_1} & S_1 \ar[r] & 0 \\ 0 \ar[r] & K_{j}\ar[r]^{f_j} & M_i \ar[r]^{p_{j}} & S_{j} \ar[r] & 0}$$ where $j > 1$. Assume $p_{j}f_1=0$. Then there is a morphism $g : K_1 \to K_{j}$ with $f_{j}g=f_1$. Passing to the cokernels in the above diagram yields a non-zero morphism from $S_1$ to $S_{j}$, a contradiction. Thus, there is an epimorphism $K_1 \to S_j$ for all $j > 1$. Set $M_{i,1}:= K_1$. Repeating this process, for any $j \ge 1$, there is a proper subobject $M_{i,j+1}$ of $M_{i,j}$ with $\Hom_\A(M_{i,j+1}, S_p) \ne 0$ for all $p > j+1$. Therefore, we get a descending chain
$\cdots \subset M_{i,2} \subset M_{i,1} \subset M$
of proper inclusions, a contradiction.
\end{proof}

\begin{Remarks} 
(1) Observe that the fact that $\A$ is artinian is crucial. For instance, the category of finitely generated modules over $k[x]$ where $k$ is a field has a generator but is not artinian. It has infinitely many non-isomorphic simple objects indexed by the irreducible polynomials.

\medskip

\noindent (2) The fact that $\A$ has a generator is also crucial. Let $Q$ be a quiver with infinitely many vertices and no arrow and let $\A$ be the category of finite dimensional representations of $Q$ over a field $k$. Then $\A$ is artinian but has no generator. It has infinitely many non-isomorphic simple objects.
%
\end{Remarks}



\begin{Prop} \label{Lemma1} If $M$ has finite length then $M$ decomposes into a finite direct sum of indecomposable objects with local endomorphism rings. Moreover, ${\rm End}_\A(M)$ is semiperfect.
\end{Prop}

\begin{proof}
It is clear that if $M$ is of finite length, then $M$ decomposes into a finite direct sum $M = M_1 \oplus \cdots \cdot \oplus M_r$ of indecomposable objects. Any $M_i$ is again of finite length. By Fitting's lemma, any endomorphism in ${\rm End}_\A(M_i)$ is an isomorphism or is nilpotent. Therefore, we get the first part of the statement. For the second part, we refer the reader to \cite[Prop. 1.2]{LPP} or \cite[Cor. 4.4]{Krause}.
\end{proof}

Recall that a full subcategory of $\A$ is \emph{exact abelian} if it is closed under taking kernels and cokernels in the ambient category $\A$. We denote by ${\rm fl}(\A)$ the full subcategory of $\A$ of those objects of finite length. This category is exact abelian and extension-closed.

%

\begin{Lemma} \label{Lemma2} Assume that $\A$ admits a generator $M$ such that ${\rm End}_\A(M)$ is right artinian. Then $M$ is of finite length.
\end{Lemma}

\begin{proof} Observe first that for any non-zero morphism $f: X \to Y$ in $\A$, since $M$ is a generator, there exists a morphism $g: M \to X$ such that $fg \ne 0$. Therefore, we see that $\Hom_\A(M,-)$ is faithful. Assume that $M$ is not artinian. Let
$$\cdots \subset M_2 \subset M_1 \subset M_0$$
be an infinite strictly descending chain of subobjects of $M$. Using the fact that $\Hom_\A(M,-)$ is left exact and faithful, we get an infinite strictly descending chain
$$\cdots \subset \Hom_\A(M,M_2) \subset \Hom_\A(M,M_1) \subset \Hom_\A(M,M_0)$$
of right ${\rm End}_\A(M)$-submodules of ${\rm End}_\A(M)$. This contradicts the fact that ${\rm End}_\A(M)$ is right artinian.
The proof of the fact that $M$ is noetherian is similar since by the Hopkins-Levitzki theorem, the ring ${\rm End}_\A(M)$ is also right noetherian.
\end{proof}

\begin{Remark} Note that if $M$ is of finite length, then ${\rm End}(M)$ need not be right artinian (although, as we have shown, it has to be semiperfect). For instance, let $\B$ be the category of right modules over a right artinian ring $R$ that is not left artinian. Note that $R_R$ has finite length in $\B$. Consider the category $\A = \B^{\rm op}$. Now, $R_R$ also has finite length in $\A$ and its endomorphism ring is isomorphic to $R^{\rm op}$, which is left artinian but not right artinian.
\end{Remark}

\begin{Theo} \label{PropCycles0}Assume that $\A$ admits a generator $M$ such that ${\rm End}_\A(M)$ is right artinian. Then both
${\rm fl}(\A)$ and $\A$ have a projective generator, which is a direct summand of $M$.
\end{Theo}

\begin{proof}
By Lemma \ref{Lemma2}, we know that $M$ is of finite length. By Lemma \ref{PropFinite}, we know that ${\rm fl}(\A)$ has finitely many simple objects. Start with any simple object, say $S=E_0$. If $\Ext^1_\A(E_0,-)$ vanishes on all simple objects, then $E_0$ is projective in ${\rm fl}(\A)$. So assume otherwise. Let $S_1$ be a simple object with $\Ext^1_\A(E_0,S_1)\ne 0$. There is a non-split extension
$$0 \to S_1 \to E_1 \stackrel{g_1}{\longrightarrow} E_0 \to 0$$
where $E_1$ is indecomposable. Clearly, $E_1$ has a unique simple quotient $S$. In general, assume that $E_i$ for $i \ge 1$ has been constructed, is indecomposable and has a unique simple quotient $S$. If $\Ext^1_\A(E_i, -)$ vanishes on all simple objects of $\A$, then $E_i$ is projective in ${\rm fl}(\A)$. If not, let $S_{i+1}$ be a simple object with $\Ext^1_\A(E_i, S_{i+1})\ne 0$. Consider the non-split short exact sequence
$$0 \to S_{i+1} \to E_{i+1} \stackrel{g_i}{\longrightarrow} E_i \to 0$$
Let $g: E_{i+1} \to S'$ be an epimorphism with $S'$ simple. Since the sequence is non-split, $g$ factors through $g_i$ and, by induction, $S \cong S'$ is the unique simple quotient of $E_i$. Hence $E_{i+1}$ is indecomposable and has a unique simple quotient $S$. Assume that no $E_i$ is projective in ${\rm fl}(\A)$. Since $\A$ has a generator $M$ and all $E_i$ have a unique simple quotient $S$, there is an epimorphism from $M$ to $E_i$ for all $i \ge 0$. This is a contradiction since the $E_i$ have unbounded lengths and $M$ has finite length. Therefore, for each simple $S$, there is a projective object $P_S$ in ${\rm fl}(\A)$ with an epimorphism $P_S \to S$. If $S_1, \ldots, S_n$ is a complete list of the non-isomorphic simple objects of ${\rm fl}(\A)$, then $P:=\textstyle{\bigoplus}_{1 \le i \le n}P_{S_i}$ is a projective generator of ${\rm fl}(\A)$. Now, there is $m \ge 1$ with an epimorphism $M^m \to P$ which gives, by the projective property of $P$, that $P$ is a direct summand of $M^m$. This gives $M^m \cong P \oplus P'$. Now, it follows from Proposition \ref{PropFinite} that finite length objects decompose into finite direct sums of objects having local endomorphism rings. Therefore, we may use the Krull-Remak-Schmidt theorem for the above decomposition. Since the $P_{S_i}$ are all non-isomorphic, we get that $P$ is a direct summand of $M$. Since there is $r \ge 1$ with an epimorphism $P^r \to M$, we see that $P$ is a generator of $\A$. It remains to prove that $P$ is projective in $\A$. Equivalently, we need to prove that for $S$ a simple object, any epimorphism $f: X \to P_S$ splits. Since $P$ is a generator, we have an epimorphism $h: \textstyle{\bigoplus}_{i \in I}P \to X$. To prove that $f$ splits, we need to prove that $fh: \textstyle{\bigoplus}_{i \in I}P \to P_S$ splits. Let $u: P_S \to S$ be an epimorphism and, for $i \in I$, let $(fh)_i: P \to P_S$ be the restriction of $fh$ to the corresponding summand. Observe that if $u(fh)_i=0$ for all $i \in I$, then $ufh=0$, which is impossible. Therefore, there is some $i_0 \in I$ with $u(fh)_0 \ne 0$, which gives that $(fh)_0$ is an epimorphism in ${\rm fl}(\A)$ and hence, splits. This proves that $fh$ splits.
\end{proof}

\begin{Remarks} (1) Let $k$ be a field and consider $\A$ the category of finitely presented $k$-representations of the quiver $Q$ having two vertices and infinitely many arrows from one vertex to the other. Clearly, $\A$ is abelian with a generator $M$ but ${\rm End}_\A(M)$ is not right artinian. Observe that ${\rm fl}(\A)$ has no projective generator.

\medskip

\noindent (2) If $\A$ is a Hom-finite $k$-category where $k$ is a field, then ${\rm End}_\A(M)$ is always right artinian.
\end{Remarks}

Recall that $\A$ is \emph{Grothendieck} if (it is abelian and) it admits a generator, has arbitrary coproducts and filtered colimits of exact sequences are exact. The well known Gabriel-Popescu theorem \cite{GP} implies that any such category is a full subcategory of a module category.

\begin{Prop} \label{PropCycles1} Assume that $\A$ is a Grothendieck category having a generator $M$ with ${\rm End}_\A(M)$ right artinian. Then
$\A$ is a module category over a right artinian ring.
\end{Prop}

\begin{proof} We know that $M$ is of finite length by Lemma \ref{Lemma2}. It follows from Theorem \ref{PropCycles0} that $M$ has a projective direct summand $P$ which is also a generator. In order to prove the statement, it suffices to prove that $P$ is compact. Observe that for a short exact sequence
$$0 \to X \to Y \to Z \to 0,$$
if $X,Z$ are compact, then so is $Y$. Therefore, since $P$ is of finite length, it suffices to prove that any simple object $S$ is compact. Let $S$ be simple. Let $f: S \to \oplus_{i \in I}Z_i$ be a non-zero morphism and let $Z:=\oplus_{i \in I}Z_i$. For each $i \in I$, let $q_i: Z_i \to Z$ be the canonical injection. For each finite subset $J$ of $I$, let $Z_J:=\sum_{j \in J}q_j(Z_j)$. Observe that the $Z_J$ for $J$ finite form a directed system with inclusions. Moreover, we have $Z=\sum_{J \subseteq I \; \text{finite}}Z_J$. Since $\A$ is Grothendieck, we have $${\rm Im}f \cap \sum_{J \subseteq I \; \text{finite}}Z_J = \sum_{J \subseteq I \; \text{finite}} {\rm Im}f \cap Z_J.$$ Since the latter is simple, at least one summand ${\rm Im}f \cap Z_{J'}$ is non-zero and simple and has to be equal to ${\rm Im}f$. Therefore, we have ${\rm Im}f \subseteq Z_{J'}$ which proves that $f$ factors through $\oplus_{j \in J'}Z_j$.
\end{proof}

Restricting to length categories, we get the following.

\begin{Prop} \label{Prop1} Assume that $\A$ is a length category having a generator $M$ with ${\rm End}_\A(M)$ right artinian. Then
$\A$ is equivalent to the module category of the finitely generated right modules over a right artinian ring.
\end{Prop}

\begin{Lemma} \label{Lemma3} Let $\A$ be a Hom-finite abelian $k$-category. If $\A$ has a generator, then $\A$ is a length category.
\end{Lemma}

\begin{proof} Let $M$ be a generator. Since $\A$ is Hom-finite, ${\rm End}_\A(M)$ is a finite dimensional $k$-algebra and hence is (right) artinian. Since $M$ is a generator and the category is Hom-finite, any object is a quotient of a finite direct sum of copies of $M$. Thus, all objects are of finite length since $M$ is of finite length by Lemma \ref{Lemma2}.
\end{proof}

\begin{Theo} \label{Theo0} Let $\A$ be a Hom-finite abelian $k$-category. The following are equivalent.
\begin{enumerate}[$(1)$]
    \item $\A$ has a generator.
\item $\A$ has a projective generator.
\item $\A$ is equivalent to the category of finite dimensional modules over a finite dimensional $k$-algebra.
\end{enumerate}
\end{Theo}

\begin{proof} It is clear that $(2)$ implies $(1)$. By Theorem \ref{PropCycles0}, $(1)$ implies $(2)$. Clearly, $(3)$ implies $(2)$.
The fact that $(2)$ implies $(3)$ follows from Proposition \ref{Prop1} and Lemma \ref{Lemma3} by observing that for $P$ a projective generator, ${\rm End}_A(P)$ is a finite dimensional $k$-algebra and that the finitely generated modules over ${\rm End}_A(P)$ are the finite dimensional ones.
\end{proof}

\begin{Cor} Let $\A$ be an exact abelian extension-closed subcategory of a Hom-finite abelian $k$-category. Assume that $\A$ has finitely many indecomposable objects, up to isomorphism. Then $\A$ is equivalent to a module category over a finite dimensional $k$-algebra.
\end{Cor}

\begin{proof}
It is clear that $\A$ has a generator $M$ by taking the direct sum of all non-isomorphic indecomposable objects. The result now follows from Theorem \ref{Theo0}.
\end{proof}


\begin{Remarks} \label{RemarkKrause}(1) Assume that $\A$ is a length category or is Hom-finite over a field. Assume that $\A$ has a generator $M$ with $R:={\rm End}_A(M)$ right artinian. The ind completion ind$\A$ of $\A$ is a Grothendieck category. By the Gabriel-Popescu theorem, ind$\A$ is equivalent to a Serre quotient of the category Mod$R$ of right $R$-modules. Thus, $\A$ is equivalent to a Serre quotient of the category mod$R$ of finitely generated right $R$-modules. Any Serre subcategory of mod$R$ is uniquely determined by a set of simple modules of mod$R$. Let $e$ be the idempotent corresponding to these simple modules. Then $\A$ is equivalent to mod$(1-e)R(1-e)$ and thus has a projective generator.

\medskip

\noindent (2) Start instead with $\A$ Grothendieck having a generator $M$ with $R:={\rm End}_A(M)$ right artinian. Again, by the Gabriel-Popescu theorem, we have that $\A$ is equivalent to a Serre quotient of the category Mod$R$. However, it is not clear that such a quotient has to be again a module category. In general, a Grothendieck category need not be equivalent to a module category. For instance, take $B = k$ where $k$ is a field and let $\mathcal{S}$ be the subcategory of Mod$k$ of all finite dimensional $k$-vector spaces. Then $\mathcal{S}$ is a Serre subcategory of Mod$k$ and Mod$k/\mathcal{S}$ is not a module category as it has no indecomposable object.
\end{Remarks}

\section{Exact abelian extension-closed subcategories}

%
%

An object $X$ in $\A$ is called \emph{Schur} if ${\rm End}_\A(X)$ is a division ring.
Clearly, any Schur object is indecomposable and any simple object is Schur, by Schur's lemma.
Let $\A$ be a length category. In this section, we describe all exact abelian extension-closed subcategories of $\A$ in terms of their simple objects.

\medskip

Two objects $X,Y \in \A$ are \emph{Hom-orthogonal} provided $$\Hom_\A(X,Y)=0=\Hom_\A(Y,X).$$ Given a set of objects $\mathcal{O}$ in $\A$, we let $\C(\mathcal{O})$ denote the smallest exact abelian extension-closed subcategory of $\A$ containing the objects from $\mathcal{O}$. Let $\mathfrak{S}$ be the set such that an element $\mathcal{S} \in \mathfrak{S}$ is a collection of non-isomorphic Schur objects that are pairwise Hom-orthogonal. If $\T$ is an exact abelian extension-closed subcategory of $\A$, we let $S(\T)$ denote a complete set of representatives of the simple objects in $\T$. Clearly, $S(\T) \in \mathfrak{S}$. For $\mathcal{S}_1, \mathcal{S}_2 \in \mathfrak{S}$, we set $\mathcal{S}_1 = \mathcal{S}_2$ if the elements can be pairwise identified by isomorphisms.

\begin{Prop}
Assume that $\A$ is a length category. Then $\mathcal{S}$ in $\mathfrak{S}$ forms the non-isomorphic simple objects of $\C(\mathcal{S})$.
\end{Prop}

\begin{proof}
Let $\mathcal{S} \in \mathfrak{S}$. We define a full subcategory $\B$ of $\C(\mathcal{S})$ as follows. We declare that $\mathcal{S} \subseteq \B$ and $0 \in \B$. If $X$ in $\C(\mathcal{S})$ is the middle term of a short exact sequence
$$0 \to X' \to X \to S \to 0$$
with $X' \in \B$ and $S \in \mathcal{S}$, then we declare that $X \in \B$.
We prove that $\B=\C(\mathcal{S})$, from which the result will follow. It is sufficient to prove that $\B$ is closed under kernels, cokernels and extensions. Let $f: X \to Y$ be a non-zero morphism with $X,Y \in \B$. We prove by induction on $\ell(X)+\ell(Y)$ that the kernel $K$ of $f$ and the cokernel $C$ of $f$ lie in $\B$ (length is taken in $\A$). Consider the short exact sequences
$$0 \to X' \stackrel{u_X}{\longrightarrow} X \stackrel{v_X}{\longrightarrow} S_1 \to 0$$
and
$$0 \to Y' \stackrel{u_Y}{\longrightarrow} Y \stackrel{v_Y}{\longrightarrow} S_2 \to 0$$
where $S_1, S_2 \in \mathcal{S}$. Note that $\ell(X') \le \ell(X)-1$ and $\ell(Y') \le \ell(Y)-1$.  Assume first that $v_Yfu_X \ne 0$. Consider the commutative diagram
$$\xymatrix{0 \ar[r] & X' \ar[r]^{u_X}\ar[d]^{fu_X} & X \ar[r]^{v_X}\ar[d]^f & S_1 \ar[r]\ar[d] & 0\\ 0\ar[r] &  Y \ar@{=}[r] & Y \ar[r]  & 0\ar[r] & 0}$$
Set $K'$ the kernel of $fu_X$ and $C'$ its cokernel. Since $fu_X$ is a non-zero morphism and $\ell(X')+\ell(Y)<\ell(X) + \ell(Y)$, by induction, $K',C'$ lie in $\B$. Assume as a first case that the induced morphism $h: S_1 \to C'$ is non-zero. Now, $C$ is the cokernel of $h$ where $\ell(S_1) + \ell(C') < \ell(X)+ \ell(Y)$. Therefore, $C$ has to be in $\B$ by induction. Let $Z$ denote the kernel of $h$. Again, we know that $Z$ lies in $\B$. If $Z$ is not in $\mathcal{S}$ and is non-zero, then by definition of $\B$, there is a proper subobject $Z'$ of $Z$ which is in $\B$. Since the length of $Z$ is finite, we see that $Z$ has to have a proper subobject in $\mathcal{S}$, and hence that $S_1$ has to have a proper subobject in $\mathcal{S}$, which contradicts that $\mathcal{S}$ is Hom-orthogonal. Therefore, $Z=0$ and $h$ is a monomorphism. Thus, $K \cong K' \in \B$. Assume now that $h = 0$. Then $K$ is an extension of $S_1$ by $K' \in \B$ so $K \in \B$ by definition. Similarly, we get $C \in \B$. So assume that $v_Yfu_X = 0$. We get a commutative diagram
$$\xymatrix{0 \ar[r] & X' \ar[r]^{u_X}\ar[d]^{f'} & X \ar[r]^{v_X}\ar[d]^f & S_1 \ar[r]\ar[d]^{f''} & 0\\ 0\ar[r] &  Y' \ar[r]^{u_Y} & Y \ar[r]^{v_Y}  & S_2\ar[r] & 0}$$
If $f''$ is non-zero, then it needs to be an isomorphism. Therefore, we have $K \cong K'$, $C\cong C'$ and, by induction, $K,C \in \B$. Otherwise, $f''=0$. Either $K \cong K' \in \B$ or else, $K$ is an extension of $S_1$ by $K'\in \B$ so $K \in \B$. Similarly, $C \in \B$. It remains to prove that $\B$ is closed under extensions. Consider a short exact sequence
$$0 \to U \to V \to W \to 0$$
where $U,W \in \B$. We prove by induction on the length of $V$ that $V \in \B$. If $W$ is in $\mathcal{S}$, then we are done. Otherwise, $W$ has a proper subobject $W'$ in $\B$ with corresponding quotient an object $S \in \mathcal{S}$. Consider the pullback $E$ of the inclusion $W' \to W$ and the morphism $V \to W$. We have a short exact sequence
$$0 \to U \to E \to W' \to 0.$$
Since $\ell(E)= \ell(U) + \ell (W') < \ell (U) + \ell(W) = \ell(V)$, by induction, we have that $E \in \B$. Now, the short exact sequence
$$0 \to E \to V \to S \to 0$$
yields $V \in \B$.
\end{proof}

The following result follows from the last proposition.

\begin{Theo}
Assume that $\A$ is an abelian length category. Then there is a one-to-one correspondence between $\mathfrak{S}$ and the exact abelian extension-closed subcategories of $\A$. If $\mathcal{S} \in \mathfrak{S}$, then $\C(\mathcal{S})$ is the corresponding exact abelian extension-closed subcategory. If $\T$ is exact abelian extension-closed, then $S(\T)$ is the corresponding element in $\mathfrak{S}$.
\end{Theo}

A full subcategory $\B$ of $\A$ is \emph{thick} if it is closed under direct summands, under extensions, under kernels of epimorphisms and cokernels of monomorphisms. Clearly, if $\B$ is exact abelian extension-closed, then $\B$ is thick. The converse is not true. However, if $\A$ is hereditary, thick is equivalent to being exact abelian and extension-closed; see \cite{IPT}, for instance. Hence, we get the following.

\begin{Theo}
Assume that $\A$ is a hereditary abelian length category. Then there is a one-to-one correspondence between $\mathfrak{S}$ and the thick subcategories of $\A$. If $\mathcal{S} \in \mathfrak{S}$, then $\C(\mathcal{S})$ is the corresponding thick subcategory. If $\T$ is thick, then $S(\T)$ is the corresponding element in $\mathfrak{S}$.
\end{Theo}

\begin{Remark}
Note that the assumption of $\A$ being a length category is essential. If $\A$ is not artinian or not noetherian, then ${\rm fl}(\A)$ is exact abelian extension-closed and has the same simple objects at the ones of $\A$ but ${\rm fl}(\A) \ne \A$. Therefore, the exact abelian extension-closed subcategories are determined by their simple objects if and only if $\A$ is a length category.
\end{Remark}

\section{Hereditary categories with generators}

Let $\A$ be a Hom-finite hereditary abelian $k$-category where $k$ is an algebraically closed field. If $\A$ has a generator, then by Theorem \ref{Theo0}, we know that $\A$ is equivalent to the module category of a finite dimensional algebra. Since $\A$ is hereditary and $k = \bar k$, this yields $\A \cong \rep(Q)$ for some finite acyclic quiver $Q$.
On the other hand, if $Q$ is a finite acyclic quiver, then the category $\A:=\rep(Q)$ of finite dimensional representations of $Q$ is an hereditary abelian $k$-category and is a length category. Hence, all the results obtained so far apply.

\medskip

Assume now that $Q$ is a finite acyclic quiver having $n$ vertices $\{1,2,\ldots,n\}$. To each $M \in \rep(Q)$, we can associate its dimension vector $d_M \in (\Z_{\ge 0})^n$ such that, for $1 \le i \le n$, the $i$-th entry of $d_M$ is the dimension over $k$ of $M(i)$. The dimension vector of a Schur object in $\rep(Q)$ is called a \emph{Schur root}. Schur roots are extensively studied in geometric representation theory. Let $d=(d_1, \ldots, d_n)\in (\Z_{\ge 0})^n$. Consider $\rep(Q,d)$ the space of all representations $M$ with $M(i)=k^{d_i}$. We can consider the full subcategory $\A(d)$ of $\rep(Q)$ with
$$\A(d) = \{X \in \rep(Q) \mid \Hom(X,N)=0=\Ext^1(X,N)\; \text{for some}\;N\in\rep(Q,d)\}.$$
It is proven in \cite{PW} that this subcategory is thick and that it has a projective generator if and only if $d=d_V$ for some $V$ with $\Ext^1(V,V)=0$. This subcategory has the feature that if $X$ is a simple object of it with $\Ext^1(X,X)\ne 0$, then there are infinitely many non-isomorphic simple objects with dimension vector $d_X$ in $\A(d)$. 


\begin{Theo} The following are equivalent.
\begin{enumerate}[$(1)$]
    \item The category $\A(d)$ has a generator,
\item The category $\A(d)$ has a projective generator,
\item There is a finite acyclic quiver $Q'$ with $\A(d) \cong \rep(Q')$,
\item We have that $d$ is the dimension vector of some $V$ with $\Ext^1(V,V)=0$.
\end{enumerate}
\end{Theo}

\begin{proof}
If $\A(d)$ is equivalent to a category of finite dimensional modules over a finite dimensional $k$-algebra $A$, then $A$ has to be hereditary. Therefore, $A \cong kQ'$ for some finite acyclic quiver $Q'$, meaning that $\A(d) \cong \rep(Q')$. Thus, the equivalence of the first three statements follow from Theorem \ref{Theo0}. The equivalence of $(2)$ and $(4)$ follows from \cite{PW}.
\end{proof}

\begin{Exam} Let $Q$ be the Kronecker quiver, that is, the quiver with two vertices and two arrows pointing in the same direction. Let $d = (1,1)$. The category $\A(d)$ is the full subcategory of regular representations of $Q$. The simple objects of $\A(d)$ are indexed by $\mathbb{P}^1(k)$ and all have dimension vector $(1,1)$. It is not hard to check that if $V \in \rep(Q,d)$, then $\Ext^1(V,V) \ne 0$. It follows from the last theorem that $\A(d)$ has no generator and no projective generator. Any subset of $\mathbb{P}^1(k)$ will give rise to a thick subcategory of $\rep(Q)$ contained in $\A(d)$. In this special example, the simple objects of $\A(d)$ are all the Schur objects of $\A(d)$. Therefore, the thick subcategories of $\A(d)$ are indexed by the subsets of $\mathbb{P}^1(k)$. Since any simple object in $\A(d)$ has a self-extension, the only thick subcategory of $\A(d)$ that has a generator is the trivial one coming from $\emptyset \subseteq \mathbb{P}^1(k)$.
\end{Exam}

\begin{Exam} Let $Q=(Q_0, Q_1)$ be the infinite quiver as follows. Its underlying graph is a binary tree where all vertices but one, say $a$, have weight $3$. We choose the orientation of $Q$ so that $Q$ has a unique source vertex $a$ and all vertices but $a$ have one incoming arrow and two outgoing arrows. We consider the category $\rep^+(Q)$ of finitely presented representations of $Q$. This is a Hom-finite hereditary abelian $k$-category; see \cite{}. Consider the projective representation $P_a$ at $a$. We have $P_a(x)=k$ for all $x \in Q_0$ and $P(\alpha)=1$ for all $\alpha \in Q_1$. Then $P$ is neither noetherian nor artinian. By Lemma \ref{Lemma3}, $\rep^+(Q)$ has no generator and no projective generator. Note that $\rep^+(Q)$ has enough projective objects, though.
\end{Exam}

\begin{Exam} Let $Q$ be any infinite quiver and let $\A:= {\rm Rep}(Q)$ the category of all representations of $Q$. This is a Grothendieck abelian $k$-category (but not Hom-finite). It clearly has a projective generator $P$, however, $P$ is not compact. In fact, any projective generator is not compact. Thus, ${\rm Rep}(Q)$ is not equivalent to a module category.
\end{Exam}

\end{document}